\documentclass[12pt,reqno]{amsart}
\usepackage{amsmath,amssymb,verbatim,comment,color}
\usepackage[pagebackref,pdftex,hyperindex]{hyperref}

\usepackage[margin=3cm]{geometry}
\usepackage{color}
\newcommand{\C}{\mathbb{C}}

\renewcommand{\P}{\mathbb{P}}
\newcommand{\Q}{\mathbb{Q}}
\newcommand{\R}{\mathbb{R}}

\newcommand{\ft}{\mathfrak{t}}

\newcommand{\cH}{\mathcal{H}}

\newcommand{\cX}{\mathcal{X}}

\renewcommand{\a}{\alpha}

\renewcommand{\d}{\delta}
\newcommand{\e}{\varepsilon}

\renewcommand{\r}{\rho}

\newcommand{\bp}{\bar{\partial}}
\newcommand{\dd}{\sqrt{-1}\partial \bar{\partial}}
\newcommand{\ddt}{\frac{d}{dt}}
\newcommand{\cf}{{\rm cf.\ }} 
\newcommand{\eg}{{\rm e.g.\ }} 
\newcommand{\ie}{{\rm i.e.\ }} 

\newcommand{\NA}{\mathrm{NA}}

\DeclareMathOperator{\DF}{DF}
\DeclareMathOperator{\Fut}{Fut}

\DeclareMathOperator{\Ric}{Ric}

\renewcommand{\leq}{\leqslant}
\renewcommand{\geq}{\geqslant}

\renewcommand{\hat}{\widehat}
\renewcommand{\tilde}{\widetilde}

\numberwithin{equation}{section}       % Number formulas within sections

\newtheorem{prop} {Proposition} [section]
\newtheorem{thm}[prop] {Theorem} 
\newtheorem{dfn}[prop] {Definition}

\newtheorem{cor}[prop]{Corollary}

\theoremstyle{remark}
\newtheorem*{ackn}{\bf{Acknowledgment}}

\title[Quantitative estimates for optimal degenerations]{The K\"ahler-Ricci flow and quantitative bounds for Donaldson-Futaki invariants of optimal degenerations} 
\date{\today}

\author[R. Takahashi]{Ryosuke Takahashi}
\address{Research Institute for Mathematical Sciences\\
 Kyoto University\\
 Kyoto 606-8502\\
 JAPAN}
\email{tryosuke@kurims.kyoto-u.ac.jp}

\subjclass[2010]{Primary 53C55; Secondary 14L24}
\keywords{K\"ahler-Ricci flow, optimal degeneration, greatest Ricci lower bound}

\begin{document}
\maketitle
\begin{abstract}
We establish a lower bound for the Donaldson-Futaki invariant of optimal degenerations produced by the K\"ahler-Ricci flow in terms of the greatest Ricci lower bound on arbitrary Fano manifolds. As an application, we can generalize the finiteness of the Futaki invariants on K\"ahler-Ricci solitons obtained by Guo-Phong-Song-Sturm to the space of all Fano manifolds. Also, we discuss the relation to Hisamoto's inequality for the infimum of the $H$-functional.
\end{abstract}
%==============Section 1==========================
\section{Introduction} \label{int}
A central question in K\"ahler geometry is which Fano manifolds admit K\"ahler-Einstein metrics. More precisely, Yau-Tian-Donaldson conjecture states that a Fano manifold  admits a K\"ahler-Einstein metric if and only if it is K-polystable, \ie the Donaldson-Futaki invariant $\DF(\cX)$ is non-negative for all special degenerations $\cX$ and equality holds if and only if $\cX$ is product. This conjecture was resolved by Chen-Donaldson-Sun \cite{CDS15a,CDS15b,CDS15c} and Tian \cite{Tia15}.

It seems that not so much things are known in unstable cases. Let $X$ be an $n$-dimensional Fano manifold. According to \cite{He16}, we define the $H$-functional to be
\[
H(\omega):=\int_X \r e^\r \omega^n, \quad \omega \in c_1(X),
\]
where $\r$ is the Ricci potential of $\omega$ defined by
\[
\Ric(\omega)-\omega=\dd \r, \quad \int_X e^\r \omega^n=(-K_X)^n=:V.
\]
By Jensen's inequality we have $H(\omega) \geq 0$ with equality holds if and only if $\omega$ is K\"ahler-Einstein. A concern in the unstable case is to construct a test configuration which optimally destabilizes the K\"ahler or algebro-geometric structure of $X$. Geometric flows are one of the most effective tool to attack this problem. Indeed, using the resolution to the Hamilton-Tian conjecture \cite{CW14}, Dervan-Sz\'ekelyhidi \cite{DS16b} showed that the destabilizer $\cX_a$ produced by the K\"ahler-Ricci flow \cite{CSW18} is optimal in the sense that
\begin{equation} \label{opt}
\inf_{\omega \in c_1(X)}H(\omega)=\sup_\cX (-H(\cX))=-H(\cX_a),
\end{equation}
where the supremum of the RHS is taken over all special degenerations. The $H$-invariant $H(\cX_a)$ they introduced is computed on a (singular) K\"ahler-Ricci soliton, which arises as the unique sequential polarized Gromov-Hausdorff limit along the flow (see Section \ref{pre} and Section \ref{plm} for precise details). However a problem is that it is hard to compute $H(\cX_a)$ and $\DF(\cX_a)$ directly since the soliton vector field is determined implicitly as the (unique) critical point of a strictly convex function on the Lie algebra consisting of holomorphic vector fields (\cf \cite{BN14}).

So it has a meaning to establish the ``quantitative'' bounds for $H(\cX_a)$ and $\DF(\cX_a)$, that is, we want to get bounds by explicitly computable numbers. We consider the greatest Ricci lower bound of $X$ \cite{Sze11} defined by
\[
R(X):=\sup \bigg\{ r \in [0,1] \bigg| \text{${}^\exists \omega \in c_1(X)$ such that $\Ric(\omega)>r\omega$} \bigg\}.
\]
It is known that the equality $R(X)=1$ holds if and only if $X$ is K-semistable (\cf \cite{BBJ15,Li17}), and the invariant $R(X)$ is related to Berman-Fujitas' $\d$-invariant $\d(X)$ by the formula $R(X)=\min\{ \d(X),1\}$ \cite{BJ18,Fuj16}. Also, we have $R(X) \leq S(X)$, where $S(X)$ is an algebro-geometric invariant defined in terms of uniform twisted K-stability (see \cite[Corollary 1.5]{Der16}). A remarkable thing is that the invariant $R(X)$ is computed explicitly in many specific cases (\eg \cite{Li11,Del17}).

The main theorem in this paper is the following:
\begin{thm} \label{bdf}
For any $n$-dimensional Fano manifold $X$, we have
\[
\DF(\cX_a) \geq -\frac{1-R(X)}{R(X)}nV,
\]
where the Donaldson-Futaki invariant $\DF(\cX_a)$ is computed on the $\Q$-Fano variety admitting a (singular) K\"ahler-Ricci soliton $(Y,\omega_Y,W_Y)$, which arises as the unique sequential polarized Gromov-Hausdorff limit along the K\"ahler-Ricci flow starting from any K\"ahler metric in $c_1(X)$. Also, the algebraic invariant $\DF(\cX_a)$ coincides with the integral invariant
\[
\Fut(W_Y):=\int_Y |\nabla \r_Y|_{\omega_Y}^2 \omega_Y^n=\int_Y |W_Y|_{\omega_Y}^2 \omega_Y^n,
\]
where $\r_Y$ denotes the Ricci potential of $\omega_Y$.
\end{thm}
It seems that Theorem \ref{bdf} is new even in the case when $Y$ is isomorphic to $X$ (which is equivalent to say that $X$ admits a K\"ahler-Ricci soliton by \cite[Corollary 4.3]{DS16b}). To prove Theorem \ref{bdf}, we study the limit space $(Y,\omega_Y,W_Y)$ by mean of the non-Archimedean limits of energy functionals on the space of K\"ahler metrics developed in \cite{BHJ17,BHJ19}.

Theorem \ref{bdf} has some applications. First, by \cite{Cam92,KMM92}, we know that there exists uniform constant $C=C(n)>0$ (which is independent of $X$) such that $V<C$. Also from the uniform positive lower bound of the log canonical threshold \cite{Bir16}, there exists a uniform constant $\e=\e(n)>0$ (which is independent of $X$) such that $R(X)>\e$ (see \cite[Corollary 2.1]{GPSS18}). Combining with Theorem \ref{bdf} we obtain the following:
\begin{cor} \label{fss}
In the same setting as in Theorem \ref{bdf}, there exists a uniform constant $F=F(n)>0$ such that for any $n$-dimensional Fano manifold $X$, we have
\begin{equation} \label{bifv}
\Fut(W_Y) \geq -F.
\end{equation}
\end{cor}
In particular, when $X$ admits a K\"ahler-Ricci soliton $(\omega_X,W_X)$, one can apply the argument in Section \ref{plm} to $X$ directly, and get
\[
\Fut(W_X) \geq -F,
\]
which was conjectured in \cite[page 31]{PSS15}, and thereafter solved by Guo-Phong-Song-Strum \cite[Corollary 1.1]{GPSS18} by showing the compactness of K\"ahler-Ricci solitons. So Corollary \ref{fss} can be regarded as a generalization of their result to the space of all Fano manifolds. Also we should remark that the property \eqref{bifv} does not hold for the space of all (singular) K\"ahler-Ricci solitons on $\Q$-Fano varieties with at worst log terminal singularities as discussed in \cite[page 31]{PSS15}.

Now we will explain another application. We have the following inequality \eqref{har} as a direct consequence from \eqref{opt}, Theorem \ref{bdf} and $\DF(\cX_a) \leq H(\cX_a)$ (\cf \cite[Lemma 2.5]{DS16b}):
\begin{cor} \label{ubound}
For any $n$-dimensional Fano manifold $X$, we have
\begin{equation} \label{har}
\inf_{\omega \in c_1(X)} H(\omega) \leq \frac{1-R(X)}{R(X)}nV.
\end{equation}
\end{cor}
Corollary \ref{ubound} is inspired by Hisamoto's inequality \cite[Proposition 5.4]{His19}: if $R(X)>1/4\pi$, we have
\begin{equation} \label{his}
\inf_{\omega \in c_1(X)} H(\omega) \leq (1-R(X))nV.
\end{equation}
Hisamoto's proof is quite different from ours. Indeed, he proved the inequality \eqref{his} by using the relation between $\inf_\omega H(\omega)$ and the supremum of Perelman's $\mu$-functional (based on \cite[Theorem 4.2]{DS16b}), and applying the log-Sobolev inequality. Although our inequality \eqref{har} is weaker than \eqref{his}, an advantage is that it holds without any restrictions for $R(X)$.
\begin{ackn}
The author expresses his gratitude to Prof. Tomoyuki Hisamoto for useful discussions on this paper. Also the author is grateful to the referees for insightful comments which have helped to improve the article.
\end{ackn}
%==============Section 2==========================
\section{Preriminaries} \label{pre}
Fix a K\"ahler metric $\hat{\omega} \in c_1(X)$. We start with the definitions and properties of several functionals on the space of K\"ahler potentials
\[
\cH:=\{ \phi \in C^\infty(X;\R)|\omega_\phi:=\hat{\omega}+\dd \phi>0\}.
\]
The Mabuchi functional, or K-energy is defined by its variation
\[
\d M(\phi):=-\int_X \d \phi (S_\phi-n)\omega_\phi^n,
\]
where $S_\phi$ denotes the scalar curvature of $\omega_\phi$. The $I$ and $J$-functionals are defined by
\[
I(\phi):=\int_X \phi(\hat{\omega}^n-\omega_\phi^n),
\]
\[
\d J(\phi):=\int_X \d \phi (\hat{\omega}^n-\omega_\phi^n).
\]
For $\phi \in \cH$ we define the Ricci potential $\r_\phi$ by
\[
\Ric(\omega_\phi)-\omega_\phi=\dd \r_\phi, \quad \int_X e^{\r_\phi} \omega_\phi^n=V,
\]
and the $H$-functional by
\[
H(\phi):=\int_X \r_\phi e^{\r_\phi} \omega_\phi^n.
\]

Next we review the non-Archimedean limits of $M$, $I$ and $J$ in the sense of \cite{BHJ17,BHJ19}. Here we will restrict our selves to special degenerations \cite{Tia97}:
\begin{dfn}
A special degeneration of $X$ is a flat normal $\Q$-Fano family $\pi \colon \cX \to \C$, together with a holomorphic vector field $v$ on $\cX$, which generates a $\C^\ast$-action on $\cX$ covering the standard action on $\C$. In addition, the fiber $\cX_\tau$ is required to be isomorphic to $X$ for one, and hence all $\tau \in \C^\ast$.
\end{dfn}
For a special degeneration $\cX$, we consider the corresponding Hilbert and weight polynomials
\[
N_k=a_0 k^n+a_1 k^{n-1}+O(k^{n-2}),
\]
\[
w_k=b_0 k^{n+1}+b_1 k^n+O(k^{n-1}).
\]
Then the {\it Donaldson-Futaki invariant} is defined by
\[
\DF(\cX):=n!(2b_1-nb_0).
\]
Also we define
\[
(I^{\NA}-J^{\NA})(\cX):=n!(-b_0+\lambda_{\rm max} a_0).
\]
where $\lambda_{\rm max} \in \Q$ denotes the maximum weight of the $\C^\ast$-action on $\cX_0$ (indeed, the functional $I$, $J$ admits its own non-Archimedean limit $I^{\NA}$, $J^{\NA}$ respectively, but we will not use their definitions here). The functional $I^{\NA}-J^{\NA}$ is also called {\it minimum norm} of $\cX$ (\cf \cite[Remark 7.12]{BHJ17}, \cite[Section 2.1]{Der16}). Let us consider also the analytic aspects of these functionals. We choose a {\it smooth} K\"ahler metric $\omega_0 \in c_1(\cX_0)$, \ie it is the restriction of a smooth metric under a projective embedding of $\cX_0$. We define its {\it Ricci potential} $\r_0$ by
\[
\Ric(\omega_0)-\omega_0=\dd \r_0, \quad \int_{\cX_0} e^{\r_0} \omega^n=V.
\]
So the function $\r_0$ is continuous on $\cX_0$ and smooth on $\cX_{0,{\rm reg}}$. Assume $\omega_0$ is ${\rm Im}(v)$-invariant and let $\theta_0$ be a Hamiltonian for the induced holomorphic vector field $v$ on $\cX_0$. Then we have
\[
\DF(\cX)=-\int_{\cX_0} \theta_0 e^{\r_0} \omega_0^n+\int_{\cX_0} \theta_0 \omega_0^n,
\]
\[
(I^{\NA}-J^{\NA})(\cX)=-\int_{\cX_0} \theta_0 \omega_0^n+V \max_{\cX_0} \theta_0.
\]
Actually, to prove this, we take a resolution $p \colon \tilde{\cX}_0 \to \cX_0$ and compute the algebraic Futaki invariant on $\tilde{\cX}_0$ with respect to the polarization $p^\ast(-kK_{\cX_0})$ for a large divisible $k$, using the (equivariant) Riemann-Roch formula by a smooth background metric (\cf \cite[page 264]{CDS15c}). In order to deal with the maximum weight $\lambda_{\rm max}$, we should take $\omega_0=k^{-1} \omega_{\rm FS}$ so that $v$ has the Hamiltonian function
\[
\theta_0=\frac{\sum_i \lambda_i |Z_i|^2}{\sum_i |Z_i|^2}
\]
for some suitable homogeneous coordinates $\{Z_i\}$ (as in the proof of \cite[Lemma 12]{DS16a}). Since $\cX_0$ is not contained in any hyperplanes, we can take $x_0 \in \cX_0 \backslash \{Z_{\rm max}=0\}$ and consider the gradient flow of $\theta_0$ starting from $x_0$ on $\P^{N_k-1}$. Since $\cX_0$ is compact and invariant under the Hamiltonian action of $v$, the flow converges to a limit point $x_\infty \in \cX_0$ where
\[
\max_{\cX_0} \theta_0 \geq \theta_0(x_\infty)=\lambda_{\rm max}=\max_{\P^{N_k-1}} \theta_0.
\]
Then a consequence from \cite{BHJ17,BHJ19} is the following:
\begin{thm}
Let $\cX$ be a special degeneration of $X$ and $\{\phi_t\} \subset \cH$ be an associated geodesic ray, then we have
\[
\lim_{t \to \infty} \frac{M(\phi_t)}{t}=\DF(\cX), \quad \lim_{t \to \infty} \frac{(I-J)(\phi_t)}{t}=(I^\NA-J^\NA)(\cX).
\]
\end{thm}
On the other hand, the $H$-functional does not have such an non-Archimedean description for a certain energy functional. However according to \cite{DS16b}, we define the {\it $H$-invariant} $H(\cX)$ to be
\[
H(\cX):=2n!b_1-(n+1)!b_0+V \lim_{k \to \infty} \log \bigg( \frac{1}{N_k} \sum_{i=1}^{N_k} e^{\frac{\lambda_{k,i}}{k}}\bigg),
\]
where $\lambda_{k,i}$ denotes the weight of $\C^\ast$-action on $H^0(\cX_0,-kK_{\cX_0})$ generated by $v$. As shown in \cite[Proposition 2.12]{DS16b},  the $H$-invariant also has an analytic description
\[
H(\cX)=-\int_{\cX_0} \theta_0 e^{\r_0} \omega_0^n+V \log \bigg( \frac{1}{V} \int_{\cX_0} e^{\theta_0} \omega_0^n \bigg).
\]
By Jensen's inequality, one can observe that $\DF(\cX) \leq H(\cX)$ with equality if and only if $\cX$ is trivial (\cf \cite[Lemma 2.5]{DS16b}).

Finally, we remark on $\R$-degenerations, a generalization of the notion of a test configuration using the language of filtrations \cite{DS16b}. Let $Z$ be an arbitrary projective variety with ample line bundle $L \to Z$. We define the graded coordinate ring
\[
R(Z,L):=\bigoplus_{k \geq 0} H^0(Z,kL).
\]
We write $R_k=H^0(X,kL)$ for simplicity.
\begin{dfn}
An $\R$-indexed filtration $\{F^\lambda R_k\}_{\lambda \in \R}$ consists of the data satisfying for each $k$
\begin{itemize}
\item $F$ is decreasing: $F^\lambda R_k \subset F^{\lambda'} R_k$ if $\lambda \geq \lambda'$.
\item $F$ is left-continuous: $F^\lambda R_k=\cap_{\lambda'<\lambda} F^{\lambda'} R_k$.
\item $F^\lambda R_k=0$ for sufficiently large $\lambda$ and $F^\lambda R_k=R_k$ for sufficiently small $\lambda$.
\item $F$ satisfies the multipricative property:
\[
F^\lambda R_k \cdot R^{\lambda'} R_{k'} \subset F^{\lambda+\lambda'} R_{k+k'}
\]
for all $\lambda, \lambda' \in \R$ and $k, k' \geq 0$.
\end{itemize}
\end{dfn}
The {\it associated graded ring} of the filtration is defined to be
\[
{\rm gr} F^\lambda R(Z,L):=\bigoplus_{k \geq 0} \bigoplus_i F^{\lambda_{k,i}} R_k/F^{\lambda_{k,i+1}} R_k,
\]
where the $\lambda_{k,i}$ are values of $\lambda$ where the filtration of $R_k$ is discontinuous.
\begin{dfn}
An $\R$-degenration for $(Z,L)$ is a filtration of $R(Z,mL)$ for some integer $m>0$, whose associated graded ring is finitely generated.
\end{dfn}
For a given $\R$-degeneration, let $\bar{R}$ be the associated graded ring of the filtration and set $Z_0:={\rm Proj}Z_0$. Then the filtration gives a (possibly irrational) real one-parameter family which acts on $F^{\lambda_{k,i}} R_k/F^{\lambda_{k,i+1}} R_k$ by multiplying the factor $\tau^{\lambda_{k,i}}$. We may assume that we have an embedding of $Z_0$ to a projective space $\P^{N-1}$ with $N:=\dim H^0(Z,L)$. Then the real one-parameter group is given by a projective automorphisms $e^{t\Lambda}$ which preserves $Z_0$, where $\Lambda:={\rm diag}(\lambda_{1,1},\ldots, \lambda_{1,N})$ is a diagonal matrix. In addition, we also have an embedding $Z \hookrightarrow \P^{N-1}$ so that $\lim_{t \to \infty} e^{t \Lambda} \cdot Z=Z_0$ in the Hilbert scheme. By taking the closure of $\{e^{\sqrt{-1}t\Lambda}|t \in \R\}$ in $U(N)$, we obtain a real torus $T \subset U(N)$ acting on $Z_0$. So the action of $e^{t \Lambda}$ corresponds to a choice of $\xi \in \ft:={\rm Lie}(T)$. Then as discussed in \cite{CSW18}, we can take a sequence of $\C^\ast$-subgroup $\nu_\ell$ in the complexified torus $T^\C$ with $\lim_{\tau \to 0} \nu_\ell(\tau) \cdot Z=Z_0$ and $\xi_\ell \to \xi$ as $\ell \to \infty$, where $\xi_\ell$ denotes the infinitesimal generator of $\nu_\ell$. In this way, we can approximate any $\R$-degeneration by test configurations. Moreover, for $\xi \in \ft$ and $s \in \C$, we define the {\it weight character} by
\[
C(\xi,s):=\sum_{k \geq 0, \a \in \ft^\ast} e^{-sk} \a(\xi) \dim \bar{R}_{k,\a},
\]
where $\bar{R}_{k,\a}$ denotes the components of the $\C^\ast \times T^\C$-action on $\bar{R}$ defined by the $k$-grading and $T^\C$-action on $\bar{R}$. Then $C(\xi,s)$ has a Laurent series expansion as
\[
C(\xi,s)=\frac{b_0(n+1)!}{s^{n+2}}+\frac{b_1(n+2)!}{s^{n+1}}+O(s^{-n}),
\] 
where the coefficients $b_0$, $b_1$ are smooth in $\xi$, and coincide with $b_0$, $b_1$ which appear in the asymptotic expansion of the total weight $w_k$ when $\xi$ is rational. Also a transcendental term $\lim_{k \to \infty} N_k^{-1} \sum_{i=1}^{N_k} e^{\lambda_{k,i}/k}$ in the $H$-invariant admits a continuous extension for $\xi \in \ft$. So from the algebraic descriptions of each invariant, we can extend $\DF$, $I^\NA-J^\NA$ and $H$ to be continuous under this approximation procedure (see \cite[Section 2.2]{DS16b} for more details).
%==============Section 3==========================
\section{Proof of Theorem \ref{bdf}} \label{plm}
Now we give the proof of Theorem \ref{bdf}.
\begin{proof}[Proof of Theorem \ref{bdf}]
For any K\"ahler metric $\omega_0 \in c_1(X)$, we consider the normalized K\"ahler-Ricci flow starting from $\omega_0$:
\[
\ddt \omega_t=-\Ric(\omega_t)+\omega_t.
\]
The evolution of scalar curvature along the flow (for instance, see \cite{Ive93}) is given by
\[
\ddt S_{\omega_t}=\Delta_{\omega_t} S_{\omega_t}+|\Ric(\omega_t)|^2-S_{\omega_t}.
\]
The maximum principle shows that we have $S_{\omega_t} \geq \inf_X S_{\omega_0} \cdot e^{-t}$ for all positive time. We use the following, due to Chen-Wang \cite{CW14} and Chen-Sun-Wang \cite{CSW18}:
\begin{thm} \label{htconj}
Let $(X,\omega_t)$ be the solution of the K\"ahler-Ricci flow. Then the sequential polarized Gromov-Hausdorff limit of $(X, \omega_t)$ as $t \to \infty$ is a $\Q$-Fano variety $Y$, independent of choice of subsequences, which admits a (singular) K\"ahler-Ricci soliton $\omega_Y$ with soliton vector field $W_Y$. Moreover, this convergence is improved to be in $C^\infty$-Cheeger-Gromov topology away from the singular set $Y_{\rm sing}$. Assume that $W_Y \neq 0$. Then there exists a ``two-step'' $\R$-degeneration from $X$ to $Y$, \ie an $\R$-degeneration $\cX_a$ for $X$ with $\Q$-Fano central fiber $\bar{X}$, and an $\R$-degeneration $\cX_b$ for $\bar{X}$ with central fiber $Y$. The corresponding one-parameter subgroup of automorphisms on $\cX_b$ is induced by the soliton vector field $W_Y$ on $Y$.
\end{thm}
Applying Theorem \ref{htconj} to our flow $(X,\omega_t)$, we obtain $\R$-degenerations $\cX_a$ and $\cX_b$ as above. Without loss of generality, we assume that $\cX_a$ and $\cX_b$ are special degenerations (For general case, we simply approximate $\R$-degenerations by test configurations and use the continuity for $\DF$, $I^\NA-J^\NA$ and $H$). As shown in the proof of \cite[Lemma 3.4, Proposition 3.5]{CSW18}, the weight decompositions of $H^0(\bar{X},-kK_{\bar{X}})$ and $H^0(Y,-kK_Y)$ are isomorphic for all sufficiently large and divisible $k$. Hence we we have
\[
H(\cX_a)=H(\cX_b), \quad \DF(\cX_a)=\DF(\cX_b), \quad (I^\NA-J^{\rm NA})(\cX_a)=(I^\NA-J^{\rm NA})(\cX_b)
\]
by using algebraic descriptions for each invariant (where one can define the invariants $H(\cX_b)$, $\DF(\cX_b)$ and $(I^\NA-J^{\rm NA})(\cX_b)$ in an obvious way). Let $\r_Y$ be the Ricci potential of $\omega_Y$ normalized by $\int_X e^{\r_Y} \omega_Y^n=V$. Taking the trace we have
\[
S_{\omega_Y}-n=\Delta_{\omega_Y} \r_Y.
\]
Also we compute
\[
\dd(\Delta_{\omega_Y} \r_Y+\r_Y+|\bp \r_Y|_{\omega_Y}^2)=0
\]
on $Y_{\rm reg}$ by using the fact that the function $\r_Y$ is the Hamiltonian of the holomorphic vector field $W_Y$ with respect to $\omega_Y$ (for instance, we can check it in the same way as the smooth case \cite[equation (1.11)]{TZ02} since the computation is local). Since $Y$ is normal, we also get
\[
\Delta_{\omega_Y} \r_Y+\r_Y+|\bp \r_Y|_{\omega_Y}^2=c.
\]
We know that $\r_Y$ extends to a continuous function on $Y$. Also we know that $\nabla \r_Y$ is bounded and the Minkowski codimension of $Y_{\rm sing}$ strictly greater than 2. It follows that integrating by parts works well since we can take an exhaustive $K \subset Y_{\rm reg}$ with the volume of $Y \backslash K$ being small as well as we please (see the proof of \cite[Proposition 3.6]{CSW18} and \cite[Lemma 3.5, Theorem 4.2]{DS16b}). Here we note that these regularity results do not imply that the metric $\omega_Y$ is smooth in the sense of Section \ref{pre}. Nevertheless we can detect the constant $c$ as
\begin{equation} \label{hfc}
cV=\int_Y \r_Y e^{\r_Y} \omega_Y^n=-H(\cX_b),
\end{equation}
and also
\begin{equation} \label{fdf}
\DF(\cX_b)=\int_Y \r_Y \omega_Y^n+H(\cX_b),
\end{equation}
where the last equality of \eqref{hfc} was shown in the proof of \cite[Theorem 3.2]{DS16b}. To prove \eqref{fdf}, we used also the fact that the invariant $\DF(\cX_b)$ arises as the limit derivative of Ding functional along the flow generated by $W_Y$ (\cf \cite{BN14}). In particular, \eqref{hfc} and \eqref{fdf} imply that for any smooth metric $\omega_0 \in c_1(Y)$ with Hamiltonian $\theta_0$ with respect to $W_Y$ normalized by $\int_Y e^{\theta_0} \omega_0^n=V$, we have
\[
\int_Y \r_Y \omega_Y^n=\int_Y \theta_0 \omega_0^n.
\]

To deal with $\max_Y \r_Y$, we need an interpretation of Hamiltonians in terms of a lifted action on a line bundle and its holomorphic sections. Let $J$ be a complex structure on $Y_{\rm reg}$, $h_Y$ a continuous fiber metric on the $\Q$-line bundle $-K_Y$ with curvature $\omega_Y$ and $\nabla$ the compatible connection on $Y_{\rm reg}$ arising as the sequential polarized Gromov-Hausdorff limit along the K\"ahler-Ricci flow $(X,\omega_t)$ in the sense of \cite{DS14}. Also, there is a positive integer $k$ (depending only on $(X,\omega_0)$) such that we have an embedding of $Y \hookrightarrow \P^{N-1}$ by $L^2$-orthonormal sections $(s_1,\ldots,s_N)$ of $H^0(Y,-kK_Y)$ with respect to $h_Y$. We may let $k=1$ for simplicity, and further assume that $s_i$'s are eigensections with weights $\lambda_i$ with respect to the infinitesimal $V_Y$-action, where $V_Y$ denotes the real part of $W_Y$ (see also \cite[Section 3.2]{CSW18}). We know that the $JV_Y$-action on $H^0(Y,-K_Y)$ has the following expression (\cf \cite{Kob95}):
\begin{equation} \label{iac}
R(s):=\ddt \exp(tJV_Y) \cdot s|_{t=0}=\sqrt{-1}\r_Y s-\nabla_{JV_Y} s, \quad s \in H^0(Y,-K_Y)
\end{equation}
on $Y_{\rm reg}$. Since $\r_Y$ and $R(s)$ are continuous (here we used the fact that the $JV_Y$-action on holomorphic sections is compatible with the projective embedding of $Y$), we can extend the derivative in the vertical direction $\sqrt{-1}\r_Y s$ as well as horizontal direction $-\nabla_{JV_Y}s$ as a continuous section of $-K_Y$ over $Y$. This together with the continuity of $h_Y$ shows that the equality
\[
\sqrt{-1} \lambda_i |s_i|_{h_Y}^2=\sqrt{-1} \r_Y |s_i|_{h_Y}^2-(\nabla_{JV_Y} s_i, s_i)_{h_Y}
\]
holds on $Y$ (where we note that the real part of $(\nabla_{JV_Y} s_i, s_i)_{h_Y}$ actually vanishes since the function $|s_i| _{h_Y}^2$ is invariant under the $JV_Y$-action). Now we repeat the argument in Section \ref{pre} to find a point $x \in Y$ such that:
\begin{itemize}
\item $s_N(x) \neq 0$ for the eigensection $s_N$ corresponding to the maximum weight $\lambda_N=\lambda_{\rm max}$.
\item $s_i(x)=0$ whenever $\lambda_i<\lambda_N$.
\item $W_Y|_x=0$ where we regard $W_Y$ as a holomorphic vector field on $\P^{N-1}$.
\end{itemize}
Moreover, since $x$ is a fixed point of $JV_Y$-action, the derivative in the horizontal direction vanishes at $x$ (in other words, $JV_Y$ acts on the fiber of $x$ with weight $\sqrt{-1} \r_Y(x)$). So we have
\begin{equation} \label{vsp}
(\nabla_{JV_Y}s_N)(x)=0.
\end{equation}
We remark that if $x$ is a regular point of $Y$, we get \eqref{vsp} immediately from the third property $W_Y|_x=0$. In general case, it seems to be difficult to prove the continuity of $\nabla_{JV_Y}s_N$ and \eqref{vsp} without the formula \eqref{iac} since the connection $\nabla$ is intrinsic, defined only on $Y_{\rm reg}$. Anyway, we have
\[
\lambda_{\rm max}=\r_Y(x) \leq \max_Y \r_Y,
\]
which is enough to prove our statement.

By $S_{\omega_t} \geq \inf_X S_{\omega_0} \cdot e^{-t}$ and $C^\infty$-Cheeger-Gromov convergence of the K\"ahler-Ricci flow away from $Y_{\rm sing}$, we have $S_{\omega_Y} \geq 0$ on $Y_{\rm reg}$. So we have
\begin{eqnarray*}
\r_Y &\leq& S_{\omega_Y}+|\bp \r_Y|_{\omega_Y}^2+\r_Y \\
&=& \Delta_{\omega_Y} \r_Y+|\bp \r_Y|_{\omega_Y}^2+\r_Y+n\\
&=& n-\frac{1}{V} H(\cX_b)
\end{eqnarray*}
on $Y_{\rm reg}$. Since $\r_Y$ is continuous, the above inequality actually holds on $Y$. Let $R(X)$ be the greatest lower bound of the Ricci curvature on $X$. From \cite{Sze11} we know that for any $r \in (0, R(X))$ the twisted Mabuchi functional $M+(1-r)(I-J)$ is coercive. By taking the non-Archimedean limit we have
\[
\DF(\cX_a)+(1-r)(I^\NA-J^\NA)(\cX_a) \geq 0.
\]
Thus
\begin{eqnarray*}
\DF(\cX_a) &\geq& (1-r)\bigg[\int_Y \r_Y \omega_Y^n-V \max_Y \r_Y \bigg]\\
&\geq& (1-r)(\DF(\cX_a)-H(\cX_a)+H(\cX_a)-nV)\\
&=& (1-r)(\DF(\cX_a)-nV),
\end{eqnarray*}
and hence
\[
\DF(\cX_a) \geq -\frac{1-r}{r}nV.
\]
By letting $r \nearrow R(X)$, we finish the proof.
\end{proof}
%==============References==========================


\newpage
\begin{thebibliography}{widestlabel}

\bibitem[BBJ15]{BBJ15}R.~J.~Berman, S.~Boucksom and M.~Jonsson:
	\newblock \emph{A variational approach to the Yau-Tian-Donaldson conjecture}.
	\newblock \texttt{arXiv:1509.04561}

\bibitem[BHJ17]{BHJ17}S.~Boucksom, T.~Hisamoto and M.~Jonsson:
	\newblock \emph{Uniform K-stability, Duistermaat-Heckman measures and singularities of pairs}.
	\newblock Ann. Inst. Fourier {\bf 67} (2017), 743--841.

\bibitem[BHJ19]{BHJ19}
S.~Boucksom T.~Hisamoto and M.~Jonsson: 
\newblock \emph{Uniform K-stability and asymptotics of energy functionals in K\"ahler geometry}. 
\newblock J. Eur. Math. Soc. {\bf 21} (2019), 2905--2944.

\bibitem[Bir16]{Bir16}C.~Birkar: 
\newblock \emph{Singularities of linear systems and boundedness of Fano varieties}, \texttt{arXiv:1609.05543}.

\bibitem[BJ18]{BJ18}S.~Boucksom and M.~Jonsson: 
\newblock \emph{A non-Archimedean approach to K-stability}, \texttt{arXiv:1805.11160}.

\bibitem[BN14]{BN14}R.~J.~Berman and D.~W.~Nystr\"om: 
\newblock \emph{Complex optimal transport and pluripotential theory of K\"ahler-Ricci solitons}, \texttt{arXiv:1401.8264}.

\bibitem[Cam92]{Cam92}F.~Campana
\newblock\emph{Connexit\'e retionnelle des vari\'et\'es de Fano}.
\newblock Ann. Sci. \'Ecole Norm. Sup. \textbf{25} (1992), 539--545.

\bibitem[CDS15a]{CDS15a}X.X.~Chen, S.~K.~Donaldson and S.~Sun: 
\newblock\emph{Kahler-Einstein metrics on Fano manifolds, I: approximation of metrics with cone singularities}.
\newblock J. Amer. Math. Soc. \textbf{28} (2015), 183--197. 

\bibitem[CDS15b]{CDS15b}X.X.~Chen, S.~K.~Donaldson and S.~Sun:  
\newblock\emph{Kahler-Einstein metrics on Fano manifolds, II: limits with cone angle less than $2\pi$}. 
\newblock J. Amer. Math. Soc. \textbf{28} (2015), 199--234. 
    
\bibitem[CDS15c]{CDS15c}X.X.~Chen, S.~K.~Donaldson and S.~Sun:  
\newblock\emph{Kahler-Einstein metrics on Fano manifolds, III: limits as cone angle approaches $2\pi$ and completion of the main proof}.

\bibitem[CSW18]{CSW18}X.~Chen, S.~Sun and B.~Wang:
	\newblock \emph{K\"ahler-Ricci flow, K\"ahler-Einstein metric, and K-stability}.
	\newblock Geometry \& Topology {\bf 22} (2018), 3145--3173.

\bibitem[CW14]{CW14}X.~Chen and B.~Wang:
	\newblock \emph{Space of Ricci flows (II)}.
	\newblock \texttt{arXiv:1405.6797}.

\bibitem[Del17]{Del17}T.~Delcroix:
	\newblock \emph{K\"ahler-Einstein metrics on group compactifications}.
	\newblock Geometric and Functional Analysis {\bf 27} (2017), 78--129.

\bibitem[Der16]{Der16}R.~Dervan:
	\newblock \emph{Uniform stability of twisted constant scalar curvature K\"ahler metrics}.
	\newblock Int. Math. Res. Not. {\bf 2016} (2016), 4728--4783.

\bibitem[DS14]{DS14}S.~Donaldson and S.~Sun:
	\newblock \emph{Gromov-Hausdorff limits of K\"ahler manifolds and algebraic geometry}.
	\newblock Acta Math. {\bf 213} (2014), 63--228.

\bibitem[DS16a]{DS16a}V.~Datar and G.~Sz\'ekelyhidi:
	\newblock \emph{K\"ahler-Einstein metrics along the smooth continuity method}.
	\newblock Geom. Funct. Anal. {\bf 26} (2016), 975--1010.
	
\bibitem[DS16b]{DS16b}R.~Dervan and G.~Sz\'ekelyhidi:
	\newblock \emph{The K\"ahler-Ricci flow and optimal degenerations}.
	\newblock \texttt{arXiv:1612.07299}.

\bibitem[Fuj16]{Fuj16}K.~Fujita:
	\newblock \emph{A valuative criterion for uniform K-stability of $\Q$-Fano varieties}.
	\newblock \texttt{arXiv:1602.00901}, to appear in J. Reine Angew. Math.

\bibitem[He16]{He16} W. He:
	\newblock \emph{ K\"ahler-Ricci solitons and the $H$-functional}
	\newblock Asian J. of Math, {\bf 20} (2016), no. 4, 645--663.

\bibitem[His19]{His19}T.~Hisamoto:
	\newblock \emph{Geometric flow, multiplier ideal sheaves and optimal destabilizer for a Fano manifold}.
	\newblock \texttt{arXiv:1901.08480}.

\bibitem[Ive93]{Ive93}T.~Ivey:
	\newblock \emph{Ricci solitons on compact three-manifolds}.
	\newblock Diff. Geom. and its Appl. {\bf 3} (1993), 301--307.

\bibitem[KMM92]{KMM92}J.~Koll\'ar, Y.~Miyaoka and S.~Mori:
	\newblock \emph{Rational connectedness and boundedness of Fano manifolds}.
	\newblock J.~Diff. Geom. {\bf 3} (1992), 765--779.

\bibitem[Kob95]{Kob95}S.~Kobayashi:
	\newblock \emph{Transformation groups in differential geometry}.
	\newblock Springer-Verlag, Berlin, 1995.

\bibitem[Li11]{Li11}C.~Li:
	\newblock \emph{Greatest lower bound on Ricci curvature for toric Fano manifolds}.
	\newblock Adv. in Math. {\bf 226} (2011), 4921--4932.

\bibitem[Li17]{Li17}C.~Li:
	\newblock \emph{Yau-Tian-Donaldson correspondence for K-semistable Fano manifolds}.
	\newblock J. Reine Angew Math. {\bf 733} (2017), 55--85.

\bibitem[PSS15]{PSS15}D.~H.~Phong, J.~Song and J.~Strum:
	\newblock \emph{Degeneration of K\"ahler-Ricci solitons on Fano manifolds}.
	\newblock Univ. Iagel. Acta Math. {\bf 52} (2015), 29--43.
	
\bibitem[GPSS18]{GPSS18}B.~Guo, D.~H.~Phong, J.~Song and J.~Strum:
	\newblock \emph{Compactness of K\"ahler-Ricci solitons on Fano manifolds}.
	\newblock \texttt{arXiv:1805.0308}.

\bibitem[Sz\'e11]{Sze11}G.~Sz\'ekelyhidi:
	\newblock \emph{Greatest lower bounds on the Ricci curvature of Fano manifolds}.
	\newblock Compositio Math. {\bf 147} (2011), 319--331.

\bibitem[Tia97]{Tia97} G.~Tian:
	\newblock \emph{K\"ahler-Einstein metrics with positive scalar curvature}.
	\newblock Invent. Math. \textbf{130} (1997), 1--37.

\bibitem[Tia15]{Tia15} G.~Tian:
	\newblock \emph{K-stability and K\"ahler-Einstein metrics}.
	\newblock Comm. Pure Appl. Math. \textbf{68} (2015), 1085--1156.

\bibitem[TZ02]{TZ02} G.~Tian and X.~Zhu:
	\newblock \emph{A new holomorphic invariant and uniqueness of K\"ahler-Ricci solitons}.
	\newblock Comment. Math. Helv. \textbf{77} (2002), 297--325.
	
\end{thebibliography}
\end{document}